\documentclass[pdflatex,sn-mathphys-num]{sn-jnl}

\usepackage[english]{babel}
\usepackage[utf8x]{inputenc}
\usepackage[T1]{fontenc}
\usepackage[normalem]{ulem}

\usepackage{amsmath}
\usepackage{amsfonts}
\usepackage{amsthm}
\usepackage{graphicx}
\usepackage{enumerate}

\numberwithin{equation}{section}
\newtheorem{theorem}{Theorem}[section]

\newtheorem{lemma}[theorem]{Lemma}

\newtheorem{claim}[theorem]{Claim}

\theoremstyle{definition}

\theoremstyle{remark}

\newcommand{\R}{\mathbb{R}}

\newcommand{\Z}{\mathbb{Z}}

\newcommand{\supp}{\operatorname{supp}}

\title[From Fractional to Set Tilings for Pairs of Lattices] {From Fractional to Set Tilings for Pairs of Lattices} 

\author*[1]{\fnm{Darrin} \sur{Speegle}}\email{darrin.speegle@slu.edu} \affil*[1]{\orgdiv{Department of Mathematics and Statistics}, \orgname{Saint Louis University}, \orgaddress{\city{St. Louis}, \state{MO}, \country{USA}}} 

\abstract{ We generalize a theorem of Isbell asserting that every countably infinite doubly stochastic matrix has a positive generalized diagonal. As an application, we prove a support-reduction theorem for simultaneous lattice tilings. Namely, if a nonnegative measurable function bounded above by one tiles Euclidean space by translations along two full-rank lattices with integer multiplicities, then its pointwise support contains a possibly nonmeasurable set whose indicator function satisfies the same two tiling identities. The proof reduces the problem on each orbit of the group generated by the two lattices to an infinite matrix rounding theorem with integer row and column margins. This matrix theorem gives a \(0\)-\(1\) matrix with prescribed integer margins and support contained in the support of the original matrix. The result is motivated by simultaneous tiling questions arising in harmonic analysis and wavelet-set constructions. } 

\keywords{lattice tiling, simultaneous tiling, wavelet sets, doubly stochastic matrices, integer margins} 

\pacs[MSC Classification]{52C22; 42C40; 15B51; 05C70}

\begin{document}

\maketitle

\section{Introduction}

Simultaneous tiling via multiple actions on a group, often $\R^n$, is a recurring theme in tiling theory \cite{BowSpe26, DutkayHanJorgensenPicioroaga, EtkindLev2023, JM2, KP22}.
If $E\subset \mathbb R^d$ tiles by such a family of transformations, then its indicator function $\chi_E$ satisfies the corresponding functional tiling identity. A related converse is much less automatic: a general nonnegative function may tile because its values distribute mass among overlapping translates, dilates, or more general images, without any individual point of its support being forced to belong entirely to one tile. The question considered in this paper is when this fractional ambiguity can be removed.

In its simplest form, the problem asks the following. Suppose that $0 \le f \le 1$ is a measurable function which tiles exactly by translations along each of two lattices. Must the support of $f$ contain a set $E$ which also tiles by translations along both lattices? The most natural version of this question would require $E$ to be measurable, but in this paper we allow $E$ to be non-measurable. Equivalently, can one replace a fractional tiling function by a $0/1$-valued tiling function without enlarging the support? This question is especially natural in connection with wavelet-set constructions, where orthonormality conditions can often be expressed through simultaneous tiling identities and where one seeks to pass from a general nonnegative tiling function to the indicator function of a measurable frequency set; see, for example, \cite{BRS26}.

The purpose of this paper is to prove such a support-reduction result in the setting considered here. Our main theorem shows that, under the relevant integer-multiplicity hypotheses, a simultaneous tiling by a nonnegative function bounded by one contains a simultaneous tiling by a set. More formally, if a measurable function $f$ satisfies
\[
0 \leq f \leq 1
\]
and tiles with respect to translation along two lattices with integer multiplicities, then there is a possibly non-measurable set
\[
E \subset \supp(f)
\]
whose indicator function satisfies the same two tiling identities. Thus the existence of such a fractional simultaneous tile already forces the existence of an honest simultaneous tile inside its support. Here and below, $\supp(f)$ denotes the pointwise support of the chosen representative of $f$.

The proof introduces two new ideas that were absent from the related result in \cite{BRS26}. The main technical ingredient is an infinite rounding theorem for matrices with integer margins, generalizing a theorem of Isbell on positive diagonals of infinite doubly stochastic matrices \cite{Isbell1955,Isbell1962}, see Theorem \ref{col_sums}. The second additional complication is in the construction of the set itself. In \cite{BRS26}, after discarding a null set, the argument worked on dyadic-rational orbits
\[
[x]=\{2^j x+q:j\in\mathbb Z,\ q\in\mathbb Q\}
\]
with $x$ irrational. This ensured that the relevant orbit parametrizations were injective, so the fiber-matrix associated with dilations and translations had no self-identifications. In the current paper, we need to account for the fact that two lattices $\Gamma_1$ and $\Gamma_2$ may have non-trivial intersection.

Related recent work on simultaneous tiling by functions includes the following. Kolountzakis and Papageorgiou studied functions with small support that tile simultaneously with several lattices, including finite abelian group models where the problem becomes especially concrete \cite{KP22}. Related questions about supports of doubly stochastic arrays arise naturally in this finite setting \cite{EtkindLevSupport, LoukakiSmallSupport}. The present paper takes a complementary point of view: rather than asking how small the support of a simultaneous tiling function can be, we ask when the support of such a function necessarily contains a simultaneous tiling set.

The paper is organized as follows. Section~\ref{matrix} proves the matrix rounding theorem. In Section~\ref{tiling} we apply this theorem to simultaneous tilings and prove the main support-reduction result.

\section{Generalization of Isbell's Theorem}\label{matrix}

A doubly stochastic matrix $A = (a_{ij})_{i \in I, j \in J}$ is a nonnegative matrix such that the row sums and the column sums are both identically $1$, where we assume that $I$ and $J$ are countable sets. We note that if $A$ is doubly stochastic then $|I| = |J|$. The following theorem is due to Isbell \cite{Isbell1955,Isbell1962}:

\begin{theorem}
Let $A$ be a doubly stochastic matrix. There exists a bijection $\sigma:I \to J$ such that for each $i \in I$,
\[
        a_{i \sigma(i)} > 0.
\]
\end{theorem}

In this section, we prove the following generalization.

\begin{theorem}\label{col_sums}
Let $I$ and $J$ be finite or countably infinite, and let
$A=(a_{ij})_{i\in I,j\in J}$ be a matrix such that
$0\leq a_{ij}\leq 1$ for all $i\in I$ and $j\in J$. Define
\[
        r_i=\sum_{j\in J} a_{ij},
        \qquad
        c_j=\sum_{i\in I} a_{ij}.
\]
Assume that
\[
        r_i\in \Z
        \quad\text{for every } i\in I,
        \qquad
        c_j\in\Z
        \quad\text{for every } j\in J.
\]
Then there exists a $0$-$1$ matrix
$B=(b_{ij})_{i\in I,j\in J}$ such that
\[
        a_{ij}=0 \implies b_{ij}=0
\]
for all $i,j$, and such that $B$ has the same row and column sums as $A$:
\[
        \sum_{j\in J} b_{ij}=r_i
        \quad\text{for every } i\in I,
        \qquad
        \sum_{i\in I} b_{ij}=c_j
        \quad\text{for every } j\in J.
\]
\end{theorem}

\begin{proof}
By padding with zero rows and zero columns if necessary, we may assume for notational simplicity that
$I=J=\mathbb N$.

Let $S$ be the set of all matrices
$X=(x_{ij})_{i,j\in\mathbb N}$ satisfying
\[
        0\leq x_{ij}\leq 1,
        \qquad
        x_{ij}=0 \text{ whenever } a_{ij}=0,
\]
and
\[
        \sum_{j\in\mathbb N} x_{ij}\leq r_i
        \quad\text{for every } i,
        \qquad
        \sum_{i\in\mathbb N} x_{ij}\leq c_j
        \quad\text{for every } j.
\]
We view $S$ as a subset of $\ell^\infty(\mathbb N\times\mathbb N)$ with
the weak$^*$ topology induced by $\ell^1(\mathbb N\times\mathbb N)$.

The set $S$ is convex. It is also weak$^*$ compact. Indeed, it is contained
in the weak$^*$ compact unit ball of $\ell^\infty(\mathbb N\times\mathbb N)$.
The coordinate conditions are weak$^*$ closed, and the row and column
inequalities are weak$^*$ closed because, for nonnegative matrices,
\[
        \sum_{j\in\mathbb N} x_{ij}\leq r_i
\]
is equivalent to
\[
        \sum_{j\in F} x_{ij}\leq r_i
        \quad\text{for every finite } F\subseteq\mathbb N,
\]
and similarly for the column inequalities. Since
$\ell^1(\mathbb N\times\mathbb N)$ is separable, $S$ is compact and
metrizable in the weak$^*$ topology.

\begin{claim}\label{main_claim}
Every extreme point of $S$ is a $0$-$1$ matrix.
\end{claim}

Assume the claim and note that $A\in S$. By Choquet's theorem, there exists a probability measure
$\mu$ on $S$ supported on $\operatorname{ext}(S)$ such that $A$ is the
barycenter of $\mu$. Thus, for every weak$^*$ continuous affine function
$\varphi$ on $S$,
\[
        \varphi(A)=\int_S \varphi(X)\,d\mu(X).
\]

We show that $\mu$-almost every extreme point has the correct row sums.
Fix a row index $i$. For each $n$, the function
\[
        X\mapsto \sum_{j=1}^n x_{ij}
\]
is weak$^*$ continuous. Therefore
\[
        \sum_{j=1}^n a_{ij}
        =
        \int_S \sum_{j=1}^n x_{ij}\,d\mu(X).
\]
Letting $n\to\infty$ and using monotone convergence gives
\[
        r_i
        =
        \sum_{j=1}^\infty a_{ij}
        =
        \int_S \sum_{j=1}^\infty x_{ij}\,d\mu(X).
\]
But every $X\in S$ satisfies
\[
        \sum_{j=1}^\infty x_{ij}\leq r_i.
\]
Hence
\[
        \sum_{j=1}^\infty x_{ij}=r_i
\]
for $\mu$-almost every $X\in S$.

Since there are only countably many rows, it follows that for $\mu$-almost
every $X\in S$,
\[
        \sum_{j=1}^\infty x_{ij}=r_i
        \quad\text{for every } i.
\]
The same argument applied to columns shows that, for $\mu$-almost every
$X\in S$,
\[
        \sum_{i=1}^\infty x_{ij}=c_j
        \quad\text{for every } j.
\]

The proof then follows from the definition of $S$ and  the fact that $\mu$ is supported on $\operatorname{ext}(S)$.
\end{proof}

It remains to prove Claim \ref{main_claim}. Claim \ref{main_claim} is an integer-margin analogue of 
the extreme point theorem for doubly stochastic matrices. The perturbation argument we use is closely related to
 Mauldon's Lemma 1 \cite{Mauldon1969}, which was re-discovered by the author in his attempt to fully understand
Isbell's theorem for infinite doubly stochastic matrices \cite{Isbell1962} and the earlier positive-diagonal result \cite{Isbell1955}. See also \cite{RademacherTorielloVielma}.

\begin{proof}[Proof of Claim \ref{main_claim}]
\begin{lemma}\label{sums}
Let $I$ be finite or countably infinite, let $i_0\in I$, and let
$(\alpha_i)_{i \in I}$ be a family of real numbers strictly between $0$ and $1$.
For each $i$, define
\[
        \alpha_i^\prime = \min(\alpha_i, 1 - \alpha_i).
\]
If
\[
        \sum_{i\in I} \alpha_i \in \Z,
\]
then
\[
        \alpha_{i_0}^\prime \le \sum_{i \not= i_0} \alpha_i^\prime .
\]
\end{lemma}

\begin{proof}
For $x\in\mathbb R$, write
\[
        \|x\|_{\mathbb Z}=\inf_{n\in\mathbb Z}|x-n|
\]
for the distance from $x$ to the nearest integer. Since $0<\alpha_i<1$, we have
\[
        \|\alpha_i\|_{\mathbb Z}=\min(\alpha_i,1-\alpha_i)=\alpha_i'.
\]
The function $\|\cdot\|_{\mathbb Z}$ satisfies the triangle inequality
\[
        \|x+y\|_{\mathbb Z}\le \|x\|_{\mathbb Z}+\|y\|_{\mathbb Z}.
\]

Let
\[
        m=\sum_{i\in I}\alpha_i\in\mathbb Z.
\]
Then
\[
        \alpha_{i_0}=m-\sum_{i\neq i_0}\alpha_i.
\]
Therefore, using invariance under integer translations and sign changes,
\[
        \alpha_{i_0}'
        =\|\alpha_{i_0}\|_{\mathbb Z}
        =
        \left\|\sum_{i\neq i_0}\alpha_i\right\|_{\mathbb Z}.
\]

If $I$ is finite, the triangle inequality gives
\[
        \left\|\sum_{i\neq i_0}\alpha_i\right\|_{\mathbb Z}
        \le \sum_{i\neq i_0}\|\alpha_i\|_{\mathbb Z}
        =
        \sum_{i\neq i_0}\alpha_i',
\]
as desired.

If $I$ is infinite, choose an increasing sequence of finite sets
$F_n\subset I\setminus\{i_0\}$ with
\[
        \bigcup_{n=1}^\infty F_n=I\setminus\{i_0\}.
\]
For each $n$, the triangle inequality gives
\[
        \left\|\sum_{i\in F_n}\alpha_i\right\|_{\mathbb Z}
        \le \sum_{i\in F_n}\alpha_i'.
\]
Since $\sum_i\alpha_i<\infty$ and $\alpha_i'\leq \alpha_i$, we may pass to the
limit as $n\to\infty$ and obtain
\[
        \left\|\sum_{i\neq i_0}\alpha_i\right\|_{\mathbb Z}
        \le \sum_{i\neq i_0}\alpha_i'.
\]
Hence
\[
        \alpha_{i_0}'\le \sum_{i\neq i_0}\alpha_i',
\]
which proves the lemma.
\end{proof}

Continuing the proof of Claim \ref{main_claim}, let $D=(d_{ij})\in S$, and suppose that $D$ is not a $0$-$1$ matrix. Write
\[
        \Omega=\{(i,j):0<d_{ij}<1\},
        \qquad
        d'_{ij}=\min\{d_{ij},1-d_{ij}\}.
\]
We must show that $D$ is not an extreme point of $S$.

Consider the bipartite graph whose vertices are the rows and
columns and whose edges are the elements of $\Omega$. Thus the edge
$(i,j)$ joins row $i$ to column $j$. Let $\Gamma$ be a connected component of this graph, and fix an edge of $\Gamma$.

For a row or column $L$, let $M_L$ denote the corresponding upper bound $r_i$ or $c_j$, and write
\[
        s_L(D)=\sum_{(i,j)\in L} d_{ij}.
\]
We say that $L$ is tight if $s_L(D)=M_L$, and loose otherwise. If $L$ is tight, then $M_L\in\Z$, and the entries of $D$ outside $\Omega$ are $0$
or $1$. Hence the sum of the entries of $D$ over $L\cap\Omega$ is an
integer. Therefore Lemma~\ref{sums} implies that, for every
$(i,j)\in L\cap\Omega$,
\[
        d'_{ij}
        \leq
        \sum_{\substack{(r,s)\in L\cap\Omega\\ (r,s)\neq(i,j)}} d'_{rs}.
\]
Consequently, if $|\gamma|\leq d'_{ij}$, then we may choose numbers
$c_{rs}$, indexed by
\[
        (r,s)\in L\cap\Omega,\qquad (r,s)\neq(i,j),
\]
such that
\begin{equation}\label{eq1}
            \sum_{\substack{(r,s)\in L\cap\Omega\\ (r,s)\neq(i,j)}} c_{rs}
        =-\gamma,
        \qquad
        |c_{rs}|\leq d'_{rs}.
\end{equation}

Suppose first that $\Gamma$ contains a cycle. Choose a simple cycle and
write its edges as
\[
        (m_0,n_0),(m_1,n_1),\ldots,(m_{2N-1},n_{2N-1}),
\]
where consecutive entries, with indices taken modulo $2N$, lie alternately
in common rows and common columns. Choose
\[
        0<\beta\leq \min_{0\leq k<2N} d'_{m_k n_k}.
\]
Define $B=(b_{ij})$ by setting $b_{ij}=0$ off the cycle and
\[
        b_{m_k n_k}=(-1)^k\beta,
        \qquad k=0,1,\ldots,2N-1.
\]
Every row and every column that meets the cycle meets it in exactly two
entries, with opposite signs. Hence all row and column sums of $B$ are
zero. Also
\[
        |b_{ij}|\leq d'_{ij}
\]
for all $i,j$. Therefore
\[
        D+B\in S,
        \qquad
        D-B\in S,
\]
and $B\neq 0$. Thus $D$ is not extreme.

It remains to consider the case where $\Gamma$ is a tree. If $\Gamma$
contains two distinct loose vertices, choose the unique finite path between
them. Let its edges be
\[
        e_0,e_1,\ldots,e_N.
\]
Choose $\beta>0$ small enough that
\[
        \beta\leq d'_e
        \quad\text{for every edge } e \text{ on the path},
\]
and also small enough that the two loose endpoint constraints remain
satisfied after perturbing by $\pm\beta$. Define $B$ to be zero off the
path and to take the alternating values
\[
        \beta,-\beta,\beta,-\beta,\ldots
\]
on the path. Every internal row or column on the path has $B$-sum zero.
Only the two loose endpoint constraints may change, and $\beta$ was chosen
small enough to preserve them. Hence again
\[
        D+B\in S,
        \qquad
        D-B\in S,
\]
with $B\neq 0$.

Thus we may assume that $\Gamma$ is a tree with at most one loose vertex.
There are two cases.

First suppose that $\Gamma$ has no loose vertices. Choose an edge
$e_0=(i_0,j_0)$ of $\Gamma$ and choose
\[
        0<\beta\leq d'_{i_0j_0}.
\]
Set
\[
        b_{i_0j_0}=\beta.
\]
We now define the remaining entries of $B$ by moving outward in the tree
from the edge $e_0$. Whenever a tight row or column $L$ has a single
already assigned incident edge $e$ with value $\gamma$, choose values on
all the other edges of $L\cap\Omega$ so that their sum is $-\gamma$ and
their absolute values are bounded by the corresponding $d'_{rs}$ as in \eqref{eq1}. Continuing by increasing graph distance
from $e_0$, every row and every column in $\Gamma$ is balanced. Set
$b_{ij}=0$ off $\Gamma$. Then all row and column sums of $B$ are zero, and
\[
        |b_{ij}|\leq d'_{ij}
\]
for all $i,j$.

Finally suppose that $\Gamma$ has exactly one loose vertex, say $L_0$.
Choose an edge $e_0\in\Gamma$ incident to $L_0$, and choose $\beta>0$ such
that
\[
        \beta\leq d'_{e_0}
\]
and such that the loose constraint $L_0$ remains satisfied after changing
its sum by $\pm\beta$. Set $b_{e_0}=\beta$, set all other edges of
$\Gamma$ incident to $L_0$ equal to zero, and then move outward from
$L_0$ in the tree. At each tight row or column, balance the already
assigned parent edge by assigning values to the remaining incident edges,
using the observation above. Set $b_{ij}=0$ off $\Gamma$.

In this construction every tight row and column has $B$-sum zero. The only
possibly nonzero row or column sum of $B$ occurs at the loose vertex $L_0$,
and $\beta$ was chosen small enough so that both perturbations remain in
$S$. Also
\[
        |b_{ij}|\leq d'_{ij}
\]
for all $i,j$, and $B\neq 0$.

In every case we have found a nonzero matrix $B$ such that
\[
        D+B\in S,
        \qquad
        D-B\in S.
\]
Therefore
\[
        D=\frac12(D+B)+\frac12(D-B)
\]
is a nontrivial convex decomposition in $S$. Hence $D$ is not an extreme
point. This proves that every extreme point of $S$ is a $0$-$1$ matrix.
\end{proof}

\section{Application to Tiling}\label{tiling}

In \cite{BRS26}, it was shown that if a function tiles $\R$ simultaneously by dilations and translations, then its support contains a possibly nonmeasurable set which also tiles $\R$ via both translations and dilations. In this section, we use Theorem \ref{col_sums} to prove a similar theorem about the support of a function which simultaneously {\bf {multi-tiles}} $\R^n$ via two lattices.

\begin{theorem}\label{multitile}
Let $\Gamma_1$ and $\Gamma_2$ be full-rank lattices in $\mathbb R^n$, and
let $A,B\in\mathbb Z_{\geq 0}$. Let $0 \leq f \leq 1$ be a nonnegative,
measurable function satisfying
\[
        \sum_{\gamma\in\Gamma_1} f(x+\gamma)=A
        \qquad\text{and}\qquad
        \sum_{\gamma\in\Gamma_2} f(x+\gamma)=B
\]
almost everywhere. Then there exists a possibly non-measurable set
$E\subseteq\{x:f(x)>0\}$ such that
\[
        \sum_{\gamma\in\Gamma_1} I_E(x+\gamma)=A
        \qquad\text{and}\qquad
        \sum_{\gamma\in\Gamma_2} I_E(x+\gamma)=B
\]
almost everywhere.
\end{theorem}
\begin{proof}
Let $N_1$ and $N_2$ be null sets outside of which the two identities for $f$
hold. Put
\[
        G=\Gamma_1+\Gamma_2
\]
and define
\[
        N=\bigcup_{g\in G}((N_1\cup N_2)-g).
\]
Since $G$ is countable, $N$ is still a null set. Moreover, $N$ is
$G$-invariant. Hence
\[
        X=\R^n\setminus N
\]
is also $G$-invariant, and for every $x\in X$ we have
\[
        \sum_{\gamma\in\Gamma_1} f(x+\gamma)=A
        \qquad\text{and}\qquad
        \sum_{\gamma\in\Gamma_2} f(x+\gamma)=B.
\]
It is enough to construct $E\subset X\cap\{f>0\}$ so that the desired
indicator identities hold for every $x\in X$.

Choose a set $T$ of representatives for the $G$-orbits in $X$. Fix $t\in T$
and write
\[
        O=t+G.
\]
Let $\mathcal R_t$ denote the collection of $\Gamma_1$-cosets contained in
$O$, and let $\mathcal C_t$ denote the collection of $\Gamma_2$-cosets
contained in $O$. Thus each $R\in\mathcal R_t$ has the form $y+\Gamma_1$,
and each $C\in\mathcal C_t$ has the form $z+\Gamma_2$, with $y,z\in O$.
Since $G$ is countable, $O$ is countable, and therefore both
$\mathcal R_t$ and $\mathcal C_t$ are finite or countably infinite.

For $R\in\mathcal R_t$ and $C\in\mathcal C_t$, the intersection $R\cap C$
is nonempty. Indeed, if
\[
        R=y+\Gamma_1,
        \qquad
        C=z+\Gamma_2,
\]
with $y,z\in O$, then $y-z\in G=\Gamma_1+\Gamma_2$. Hence we may write
\[
        y-z=\gamma_1+\gamma_2
\]
with $\gamma_1\in\Gamma_1$ and $\gamma_2\in\Gamma_2$, and then
\[
        y-\gamma_1=z+\gamma_2\in R\cap C.
\]
Moreover, since $R\cap C$ is non-empty, it follows from standard facts that $R \cap C$ is a single coset of $H = \Gamma_1 \cap \Gamma_2$.

Define
\[
        a_{RC}=\sum_{x\in R\cap C} f(x).
\]
The sums are countable sums of nonnegative numbers. Since $R$ is
partitioned by the sets $R\cap C$, with $C\in\mathcal C_t$, we have
\[
        \sum_{C\in\mathcal C_t} a_{RC}
        =
        \sum_{x\in R} f(x)
        =
        A
        \quad\text{for every } R\in\mathcal R_t.
\]
Similarly,
\[
        \sum_{R\in\mathcal R_t} a_{RC}
        =
        \sum_{x\in C} f(x)
        =
        B
        \quad\text{for every } C\in\mathcal C_t.
\]
In particular, every $a_{RC}$ is finite.

The entries $a_{RC}$ need not lie between $0$ and $1$, because an
$H$-coset may contain more than one point. We therefore separate integer
and fractional parts. Set
\[
        \theta_{RC}=a_{RC}-\lfloor a_{RC}\rfloor.
\]
Then
\[
        0\leq \theta_{RC}<1.
\]
Moreover, for every $R\in\mathcal R_t$,
\[
        \sum_{C\in\mathcal C_t} \lfloor a_{RC}\rfloor
        \leq
        \sum_{C\in\mathcal C_t} a_{RC}
        =
        A.
\]
Thus the sum of the integer parts is finite, and hence
\[
        \sum_{C\in\mathcal C_t}\theta_{RC}
        =
        A-\sum_{C\in\mathcal C_t}\lfloor a_{RC}\rfloor
        \in\Z_{\geq 0}.
\]
Similarly, for every $C\in\mathcal C_t$,
\[
        \sum_{R\in\mathcal R_t}\theta_{RC}
        =
        B-\sum_{R\in\mathcal R_t}\lfloor a_{RC}\rfloor
        \in\Z_{\geq 0}.
\]

By Theorem~\ref{col_sums}, applied to the matrix
\[
        \theta=(\theta_{RC})_{R\in\mathcal R_t,\ C\in\mathcal C_t},
\]
there exists a $0$-$1$ matrix
\[
        \epsilon=(\epsilon_{RC})_{R\in\mathcal R_t,\ C\in\mathcal C_t}
\]
such that
\[
        \sum_{C\in\mathcal C_t}\epsilon_{RC}
        =
        \sum_{C\in\mathcal C_t}\theta_{RC}
        \quad\text{for every } R\in\mathcal R_t,
\]
and
\[
        \sum_{R\in\mathcal R_t}\epsilon_{RC}
        =
        \sum_{R\in\mathcal R_t}\theta_{RC}
        \quad\text{for every } C\in\mathcal C_t.
\]
We may also require
\[
        \theta_{RC}=0 \implies \epsilon_{RC}=0.
\]

Now define
\[
        n_{RC}=\lfloor a_{RC}\rfloor+\epsilon_{RC}.
\]
Then
\[
        \sum_{C\in\mathcal C_t} n_{RC}
        =
        A
        \quad\text{for every } R\in\mathcal R_t,
\]
and
\[
        \sum_{R\in\mathcal R_t} n_{RC}
        =
        B
        \quad\text{for every } C\in\mathcal C_t.
\]

It remains to choose $n_{RC}$ points from each set $R\cap C$. Let
\[
        S_{RC}=(R\cap C)\cap\{x:f(x)>0\}.
\]
We claim that $S_{RC}$ contains at least $n_{RC}$ points. Indeed,
\[
        a_{RC}=\sum_{x\in S_{RC}} f(x).
\]
If $S_{RC}$ is infinite, there is nothing to prove, since $n_{RC}$ is finite.
If $S_{RC}$ has $m$ points, then $0<f(x)\leq 1$ on $S_{RC}$ gives
\[
        a_{RC}\leq m.
\]
If $\theta_{RC}=0$, then $\epsilon_{RC}=0$, and hence
\[
        n_{RC}=a_{RC}\leq m.
\]
If $\theta_{RC}>0$, then
\[
        n_{RC}\in\{\lfloor a_{RC}\rfloor,\lceil a_{RC}\rceil\}.
\]
Since $a_{RC}\leq m$ and $m$ is an integer, it follows that
\[
        n_{RC}\leq \lceil a_{RC}\rceil\leq m.
\]
Thus $S_{RC}$ contains a subset
\[
        E_{RC}\subset S_{RC}
\]
with
\[
        |E_{RC}|=n_{RC}.
\]

Define
\[
        E_t=\bigcup_{R\in\mathcal R_t}\bigcup_{C\in\mathcal C_t}E_{RC}.
\]
The sets $R\cap C$ partition $O$, so for every $R\in\mathcal R_t$,
\[
        |E_t\cap R|
        =
        \sum_{C\in\mathcal C_t} n_{RC}
        =
        A,
\]
and for every $C\in\mathcal C_t$,
\[
        |E_t\cap C|
        =
        \sum_{R\in\mathcal R_t} n_{RC}
        =
        B.
\]

Finally, set
\[
        E=\bigcup_{t\in T}E_t.
\]
Then
\[
        E\subset X\cap\{f>0\}.
\]
If $x\in X$, then $x$ belongs to a unique $G$-orbit $O=t+G$. The coset
$x+\Gamma_1$ is one of the rows in $\mathcal R_t$, and the coset
$x+\Gamma_2$ is one of the columns in $\mathcal C_t$. Therefore
\[
        \sum_{\gamma\in\Gamma_1} I_E(x+\gamma)
        =
        |E\cap (x+\Gamma_1)|
        =
        A
\]
and
\[
        \sum_{\gamma\in\Gamma_2} I_E(x+\gamma)
        =
        |E\cap (x+\Gamma_2)|
        =
        B.
\]
Since $\R^n\setminus X=N$ is null, the two identities hold almost
everywhere.
\end{proof}

We end this section with an example of a function $f$ on $\R^2$ and three lattices such that $f$ tiles $\R^2$ via translations along each of the three lattices, and for which the support of $f$ does not contain a set which tiles by translations along all three lattices simultaneously.

Let
\[
        Q=[0,1)^2,
        \qquad
        S=[0,2)^2,
\]
and define
\[
        f=\frac12 I_S.
\]
Let
\[
        \Gamma_1=\mathbb Z(1,0)+\mathbb Z(0,2),
        \qquad
        \Gamma_2=\mathbb Z(2,0)+\mathbb Z(0,1),
\]
and
\[
        \Gamma_3=\{(m,n)\in\mathbb Z^2:m\equiv n \pmod 2\}.
\]
Then $2\mathbb Z^2$ is a sublattice of each $\Gamma_i$, and
\[
        [\Gamma_i:2\mathbb Z^2]=2
        \qquad\text{for } i=1,2,3.
\]
Since $S=[0,2)^2$ is a fundamental domain for $2\mathbb Z^2$, the function
$I_S$ tiles at level $2$ by translations along each $\Gamma_i$. Hence
\[
        f=\frac12 I_S
\]
tiles at level $1$ by translations along each of the three lattices.

We claim, however, that $S=\{x:f(x)>0\}$ does not contain a set which tiles
at level $1$ by all three lattices simultaneously. Suppose such a set
$E\subset S$ existed. For a.e. $x\in Q$, set
\[
        a=I_E(x),
        \qquad
        b=I_E(x+(1,0)),
        \qquad
        c=I_E(x+(0,1)).
\]
The $\Gamma_1$-tiling identity at $x$ gives
\[
        a+b=1,
\]
the $\Gamma_2$-tiling identity at $x$ gives
\[
        a+c=1,
\]
and the $\Gamma_3$-tiling identity at $x+(1,0)$ gives
\[
        b+c=1.
\]
But there are no numbers $a,b,c\in\{0,1\}$ satisfying these three equations.
This contradiction shows that no such set $E$ exists.

\backmatter

\section*{Statements and Declarations}

\subsection*{Competing interests}
The author declares that he has no competing interests.

\subsection*{Funding}
No funding was received for conducting this study.

\subsection*{Data availability}
Data sharing is not applicable to this article as no datasets were generated or analyzed.

\subsection*{AI Usage}
The author used ChatGPT for editorial assistance, proofreading, and suggestions during the drafting process. The author independently provided and re-checked all mathematical arguments, statements, and references, and takes full responsibility for the final manuscript.

\bibliography{refs}

\end{document}